\documentclass[12pt,a4paper]{amsart}
\usepackage[czech,english]{babel}
\usepackage{amssymb,eucal}
\usepackage[all,cmtip]{xy}
\usepackage{hyperref}

\calclayout

\renewcommand{\:}{\colon}
\renewcommand{\.}{\mskip.5\thinmuskip\relax}

\newcommand{\rarrow}{\longrightarrow}
\newcommand{\larrow}{\longleftarrow}

\newcommand{\lrarrow}{\.\relbar\joinrel\relbar\joinrel\rightarrow\.}

\DeclareMathOperator{\Hom}{Hom}

\DeclareMathOperator{\id}{id}
\DeclareMathOperator{\colim}{colim}

\newcommand{\A}{{\mathsf A}}
\newcommand{\B}{{\mathsf B}}

\newcommand{\E}{{\mathsf E}}
\newcommand{\F}{{\mathsf F}}

\DeclareMathOperator{\Sets}{\mathsf{Sets}}

\newcommand{\arity}{{\mathsf{ar}}}

\newcommand{\sop}{{\mathsf{op}}}

\newcommand{\boT}{{\mathbb T}}
\newcommand{\boZ}{{\mathbb Z}}

\newcommand{\boR}{{\mathbb R}}

\newcommand{\modl}{{\operatorname{\mathsf{--mod}}}}

\theoremstyle{plain}
\newtheorem{thm}{Theorem}[section]

\newtheorem{lem}[thm]{Lemma}
\newtheorem{prop}[thm]{Proposition}

\theoremstyle{definition}
\newtheorem{ex}[thm]{Example}
\newtheorem{exs}[thm]{Examples}
\newtheorem{rem}[thm]{Remark}

\begin{document}

\title{Exactness of direct limits for abelian categories
with an injective cogenerator}

\author{Leonid Positselski and Jan \v S\v tov\'\i\v cek}

\address{Leonid Positselski, Institute of Mathematics, Czech Academy
of Sciences, \v Zitn\'a~25, 115~67 Prague~1, Czech Republic; and
\newline\indent Laboratory of Algebraic Geometry, National Research
University Higher School of Economics, Moscow 119048; and
\newline\indent Sector of Algebra and Number Theory, Institute for
Information Transmission Problems, Moscow 127051, Russia; and
\newline\indent Department of Mathematics, Faculty of Natural Sciences,
University of Haifa, Mount Carmel, Haifa 31905, Israel}

\email{posic@mccme.ru}

\address{Jan {\v S}{\v{t}}ov{\'{\i}}{\v{c}}ek, Charles University
in Prague, Faculty of Mathematics and Physics, Department of Algebra,
Sokolovsk\'a 83, 186 75 Praha, Czech Republic}

\email{stovicek@karlin.mff.cuni.cz}

\begin{abstract}
We prove that the exactness of direct limits in an abelian category with products and an injective cogenerator $J$ is equivalent to a condition on $J$ which is well-known to characterize pure-injectivity in module categories, and we describe an application of this result to the tilting theory. We derive our result as a consequence of a more general characterization of when inverse limits in the Eilenberg--Moore category of a monad on the category of sets preserve regular epimorphisms.
\end{abstract}

\maketitle

\setcounter{tocdepth}{1}
\tableofcontents

\section*{Introduction} \label{sec:intro}
\medskip

 Grothendieck observed~\cite[N$^\circ$\,1.5]{GrToh} that an abelian
category cannot be too nice: if in a complete, cocomplete abelian
category $\A$ both the filtered colimits and cofiltered limits are exact,
then all objects of $\A$ are zero objects.
 The argument proceeds as follows: if the filtered colimits are exact
in $\A$, then, for every family of objects $A_x\in\A$ indexed by
a set $X$, the natural morphism from the coproduct to the product
$\eta\:\coprod_{x\in X}A_x\rarrow \prod_{x\in X}A_x$ is a monomorphism.
 If the cofiltered limits are exact, then this morphism is an epimorphism.
 Assuming both these properties of an abelian category $\A$,
the morphism~$\eta$ is an isomorphism.

 In particular, given an object $A\in\A$, one can choose any infinite
set $X$ and take $A_x=A$ for all $x\in X$.
 If the natural morphism $\eta_{A,X}\:A^{(X)}\rarrow A^X$ is
an isomorphism, one can use its inverse morphism $\eta_{A,X}^{-1}$ in
order to define an operation of infinite summation on the group
$E=\Hom_\A(A,A)$ of endomorphisms of $A$, that is an abelian group
homomorphism
$$
 \sum_X\:E^X\lrarrow E
$$
extending the natural finite summation map
$$
 E^X\,\supset E^{(X)}\,\lrarrow E.
$$
 Moreover, the map $\sum_Y$ is defined for every set $Y$ of
the cardinality not greater than $X$, and for any injective map
$\iota\:Y\hookrightarrow X$ the maps $\sum_X$ and $\sum_Y$ form
a commutative triangle diagram with the split monomorphism
$E^Y\rarrow E^X$ induced by~$\iota$.
 Such kind of infinite summation structure can exist on a nontrivial
commutative \emph{monoid} $E$, but when $E$ is an abelian \emph{group},
a standard cancellation trick shows that $E=0$.

 In this paper, we offer some answers to the following two questions.
 Firstly, what can be said about abelian categories $\A$ in which
the map~$\eta$ is always a monomorphism?
 Can one claim, perhaps under suitable additional assumptions, that
the filtered colimits are exact in~$\A$\,?
 Secondly, what can one do with functorial infinite summation operations
$\sum_X$ on abelian groups, if one drops or weakens the assumption of
commutativity of the diagrams related to injective maps of
sets~$\iota$\,?

 In the terminology of~\cite{GrToh}, an abelian category is said to
satisfy Ab3 if it is cocomplete, Ab4 if it is cocomplete and
the coproduct functors are exact in it, and Ab5 if it is cocomplete
and the filtered colimits are exact.
 So Ab5 implies Ab4 and Ab4 implies Ab3.
 A finer classification is offered in the book~\cite{Mit}: in
the terminology of~\cite[Section~III.1]{Mit}, an abelian category $\A$
is $C_1$ if it is cocomplete with exact coproduct functors, $\A$ is
$C_2$ if it is complete and cocomplete and the all natural morphisms
$\eta\:\coprod_{x\in X}A_x\rarrow\prod_{x\in X}A_x$ are monomorphisms in
$\A$, and $\A$ is $C_3$ if it is cocomplete with exact filtered colimits
(cf.~\cite[Theorem~III.1.9]{Mit}).
 Any complete $C_3$\+category is $C_2$ \cite[Corollary~III.1.3]{Mit},
and any $C_2$\+category is $C_1$ \cite[Proposition~III.1.1]{Mit}.

 In this paper we prove that any abelian $C_2$\+category with
an injective cogenerator is $C_3$ (\,$=$~satisfies Ab5).
 In fact, a seemingly weaker assumption is sufficient: \emph{any
complete abelian category with an injective cogenerator $J$ such that
the natural map $\eta_{J,X}\:J^{(X)}\rarrow J^X$ is a monomorphism for
all sets $X$ satisfies Ab5}.
 Here it should be pointed out that any complete abelian category with
a cogenerator is cocomplete.
 It is also worth noticing that any complete abelian category with
a set of generators and enough injective objects has an injective
cogenerator.

 Finally, we observe that if $J\in\A$ is an injective cogenerator then
the morphism $\eta_{J,X}\:J^{(X)}\rarrow J^X$ is a monomorphism if and
only if every morphism $J^{(X)}\rarrow J$ factors through $\eta_{J,X}$.
 Moreover, morphisms $f\:J^{(X)}\rarrow J$ correspond bijectively to
$X$\+indexed families of endomorphisms $f_x\:J\rarrow J$, and to every
such family one can naturally assign an endomorphism $\prod_x f_x\:
J^X\rarrow J^X$ of the object~$J^X$.
 It follows easily that every morphism $J^{(X)}\rarrow J$ factors
through~$\eta_{J,X}$ if and only if one specific such morphism, viz.,
the natural summation morphism $s_{J,X}\:J^{(X)}\rarrow J$, does
(see Theorem~\ref{thm:ab5*-via-sums}).
 Thus the morphism $\eta_{J,X}$ is a monomorphism if and only if
the summation morphism $s_{J,X}$ factors through~$\eta_{J,X}$.
 The latter condition is similar to a condition characterizing
pure-injective modules~\cite[Theorem~7.1(iv)]{JL}.

 Let us mention an application of our result to the tilting theory in
abelian categories.
 According to~\cite[Corollary~3.12]{PS}, there is a bijective
correspondence between complete, cocomplete abelian categories
$\A$ with an injective cogenerator $J$ and an $n$\+tilting object
$T\in\A$, and complete, cocomplete abelian categories $\B$ with
a projective generator $P$ and an $n$\+cotilting object $W\in\B$.
 This correspondence is accompanied by an exact category equivalence
$\E\simeq\F$ between the tilting class $\E\subset\A$ and the cotilting
class $\F\subset\B$ taking the tilting object $T$ to the projective
generator $P$ and the injective cogenerator $J$ to the cotilting
object~$W$.

 Moreover, according to~\cite[Lemma~4.3 and Remark~4.4]{PS}, both
the full subcategories $\E$ and $\F$ are closed under both the products
and coproducts in $\A$ and~$\B$, respectively.
 In particular, the equivalence between them takes powers
$W^X\in\F\subset\B$ of the $n$\+cotilting object $W$ to powers
$J^X\in\E\subset\A$ of the injective cogenerator~$J$, and
the same applies to the copowers $W^{(X)}\in\F$ and $J^{(X)}\in\E$.
 Therefore, it follows from the results of the present paper that, in
the categorical $n$\+tilting-cotilting correspondence situation
of~\cite[Corollary~3.12]{PS}, \emph{the abelian category $\A$
satisfies Ab5 if and only if the abelian group homomorphism
$\Hom_\B(\eta_{W,X},W)\:\Hom_\B(W^X,W)\rarrow\Hom_\B(W^{(X)},W)$
is surjective, and if and only if the natural summation morphism
$s_{W,X}\:W^{(X)}\rarrow W$ factors through the morphism $\eta_{W,X}\:
W^{(X)}\rarrow W^X$ for all sets~$X$}.

\subsection*{Acknowledgments}
 We thank Ji\v{r}\'{\i} Rosick\'{y} for useful discussions and especially
for bringing the paper~\cite{ARV0} to our attention, which helped us
to contemplate a nonadditive generalization of the fact that
the existence of limit terms in a monad implies Ab5$^*$.
 We are also grateful to the anonymous referee for the suggestion to
include Examples~\ref{exs:compact-hausdorff}.

 Leonid Positselski's research is supported by research plan RVO:~67985840,
by the Israel Science Foundation grant~\#\,446/15, and by the Czech Science Foundation under the grant P201/12/G028.
 Jan \v{S}\v{t}ov\'{\i}\v{c}ek's research is supported by
the Czech Science Foundation under the grant P201/12/G028.

\bigskip
\section{Preliminaries} \label{sec:prelim}
\medskip

In the introduction we are mostly concerned with when an abelian category $\A$ with products and an injective cogenerator is Ab5. In order to facilitate the discussion of our main result, it will be helpful to switch to a formally dual setting, i.e.\ to study abelian categories $\B$ with coproducts and a projective generator. The reason is that such categories admit a very convenient description as modules over additive monads; see the introduction of~\cite{PR}, \cite[\S1]{Pper} or \cite[\S6]{PS}.

\smallskip

Let us be more specific. Recall that a \emph{monad} (see e.g.\ \cite[Appendix A]{ARV} or \cite[\S6]{Wr}) on the category of sets is a triple $(\boT,\mu,\varepsilon)$ where $\boT\:\Sets\rarrow\Sets$ is an endofunctor, $\mu\: \boT\circ\boT\rarrow\boT$ and $\varepsilon\:\id\rarrow\boT$ are natural transformations, and the associativity ($\mu\circ(\boT\mu)=\mu\circ(\mu\boT)\: \boT\circ\boT\circ\boT\rarrow\boT$) and unitality ($\mu\circ(\boT\varepsilon) = \id_\boT = \mu\circ(\varepsilon\boT) \:\boT\rarrow\boT$) axioms hold.

Given a monad $\boT$ on $\Sets$, we can consider the category $\boT\modl$ of \emph{modules} over $\boT$ (perhaps a more standard terminology is the Eilenberg-Moore category of $\boT$ or the category of $\boT$-algebras, but in this paper we tend to view $\boT$ as a generalization of an associative ring, see below).
A $\boT$-module is a map $\phi\:\boT(M)\rarrow M$ in $\Sets$, satisfying also an associativity ($\phi\circ\boT(\phi)=\phi\circ\mu_M$) and a unitality ($\phi\circ\varepsilon_M=\id_M$) axiom.
A morphism of $\boT$-modules $\phi\:\boT(M)\rarrow M$ and $\psi\: \boT(N)\rarrow N$ is a map $f\: M\rarrow N$ making the obvious square commutative, i.e.\ $f\circ\phi = \psi\circ\boT(f)$. A typical example of a monad and its category of modules is as follows:

\begin{ex} \label{ex:monad}
We can interpret the category of algebras of a given signature $\Sigma$ as $\boT\modl$ as follows. Let $\Sigma$ be the signature (i.e.\ a formal set of operation symbols) and $\arity\: \Sigma\rarrow \omega$ be an arity function (e.g.\ for groups we have $\Sigma=\{\cdot,{^{-1}},1\}$ with arities $2$, $1$ and $0$, respectively). A $\Sigma$-algebra is then a set $M$ together with maps $\phi_\sigma\: M^{\arity(\sigma)}\rarrow M$, one for each operation symbol $\sigma$.

By substituting one operation into another, one can define the set $\boT(X)$ of terms in $\Sigma$ on a given set of variables $X$.
Typically, one only considers $\Sigma$-algebras satisfying equalities of the form $t_1 = t_2$, where $t_1,t_2$ are terms in the same set of variables (e.g.\ the associativity relation in the case of groups). In that case, $\boT(X)$ will be just the set of terms on $X$ modulo consequences of these equalities or, equivalently, a free algebra on the set of generators $X$. In fact, $\boT$ is naturally a functor $\Sets\rarrow\Sets$, as a map $f\: X\rarrow Y$ in $\Sets$ induces a map $\boT(f)\: \boT(X)\rarrow\boT(Y)$, which sends a term~$t$ to one in which we substitute $f(x)\in Y$ for each variable $x\in X$. We also obtain the structure of a monad, where $\mu_X\:\boT(\boT(X))\rarrow\boT(X)$ amounts to the `opening of brackets' and $\varepsilon_X\: X\rarrow\boT(X)$ interprets each variable symbol $x\in X$ as a term.

A $\Sigma$-algebra then defines a $\boT$-module $\phi\:\boT(M)\rarrow M$, where $\phi$ evaluates each term in elements of $M$. In fact, it is well-known that the structure of a $\Sigma$-algebra is equivalent to the structure of a $\boT$-module.
\end{ex}

\begin{ex} \label{ex:modules}
Let $R$ be a ring. Then a left $R$-module is an algebra with the signature $(+,-,0, r\!\cdot\!- \; (r\in R))$. The corresponding monad is given by $\boT(X)=R^{(X)}$, where $(r_x)_{x\in X} \in \boT(X)$ represents the term $\sum_{x\in X} r_x\!\cdot\!x$, and we have an isomorphism of categories $R\modl \cong \boT\modl$. We refer to \cite{PR}, \cite[\S1]{Pper} or \cite[\S6]{PS} for more details.
\end{ex}

It will be useful in the sequel that the above examples are rather representative. A $\boT$-module $M$ (with the structure map $\phi\:\boT(M)\rarrow M$) can be viewed as a set equipped with possibly infinitary operations
\[ t_M\: M^X \lrarrow M, \]
one for each set $X$ and an element $t\in\boT(X)$, where for each $m=(m_x)_{x\in X}\in M^X$, the value $t_M(m)\in M$ is defined as the image of $t$ under the composition
\begin{equation} \label{eq:term}
\boT(X) \overset{\boT(m)}\lrarrow \boT(M) \overset{\phi}\lrarrow M.
\end{equation}
Such terms can be composed---if $t\in\boT(X)$ and $s_x\in\boT(Y)$ for each $x\in X$, then the composition
\[
M^Y \overset{((s_x)_M)}\lrarrow M^X \overset{t_M}\lrarrow M
\]
is given by a term, which we formally obtain by evaluating the map
\[
\boT(X) \overset{\boT(s)}\lrarrow \boT(\boT(Y)) \overset{\mu_Y}\lrarrow \boT(Y)
\]
at $t$ (here, $s\:X\rarrow\boT(Y)$ is the map which sends $x$ to $s_x$).

We can equip $\boT(X)$ with the structure map $\phi=\mu_X\:\boT(\boT(X))\rarrow\boT(X)$. The resulting $\boT$-module is then free on the generators in the image of $X$ under $\varepsilon_X\:X\rarrow\boT(X)$. That is, any map $f\colon X\rarrow M$, where $M\in\boT\modl$ with the structure map $\phi\:\boT(M)\rarrow M$, extends uniquely to a homomorphism $\boT(X)\rarrow M$, which is defined via~\eqref{eq:term}.

As with usual module categories, the forgetful functor $\boT\modl \rarrow \Sets$ is faithful and preserves (in fact creates) categorical limits. Since every injective map (with a nonempty domain) in $\Sets$ is a section and every surjective map is a retraction, $\boT$ preserves both injectivity and surjectivity. It follows that the image of any homomorphism $f\:M\rarrow N$ of $\boT$-modules is canonically again a $\boT$-module.

The main question which we aim to study here is when an inverse limit of surjective maps of $\boT$-modules is again surjective (if $\B=\boT\modl$ is an abelian category, this is precisely the Ab5 condition for $\A=\B^\sop$). Note that there is an intrinsic characterization of surjective maps $f\:M\rarrow N$ in the category $\boT$-modules, they are so-called \emph{regular epimorphisms}. By definition, these are coequalizers of a pair of maps $M' \rightrightarrows M$, \cite[0.16]{ARV}. Indeed, every coequalizer is surjective by the previous paragraph, while every surjective map is the coequalizer of the two obvious maps $M' := M\times_NM \rightrightarrows M$ induced by the two projections $M\times M\rightrightarrows M$.

\smallskip

Finally, we turn back to abelian categories $\B$ with coproducts and a projective generator $P\in\B$. Here, we can define a monad
\[ \boT\colon X \longmapsto \Hom_\B(P,P^{(X)}), \]
where $\varepsilon_X\colon X \rarrow \Hom_\B(P,P^{(X)})$ sends $x\in X$ to the split inclusion $\iota_x\:P\rarrow P^{(X)}$ into the coproduct of $X$ copies of $P$, and where $\mu\:\boT\circ\boT\rarrow\boT$ is given by composition of maps and the universal property of the coproducts. In particular, $\boT(1) = \Hom_\B(P,P)$ is the endomorphism ring of $P$ and the ring structure can be recovered from $\boT$.

If $B\in\B$ is an arbitrary object, then $M := \Hom_\B(P,B)$ is naturally a $\boT$-module via $\phi:=\Hom_\B(P,c)\:\boT(M)\rarrow M$, where $c\:P^{(\Hom_\B(P,B))}\rarrow B$ is the canonical map. This assignment is, in fact, functorial, and we also have that

\begin{lem} \label{lem:monads}
The functor $\Hom_\B(P,-)\:\B\rarrow\boT\modl$ is an equivalence of categories. Under this equivalence, $P$ corresponds to $\boT(1)$.
\end{lem}

\begin{proof}
See e.g.\ \cite[\S6.2]{PS}.
\end{proof}

The monads $\boT$ which arise from abelian categories with coproducts and a projective generator were called \emph{additive} in~\cite{PR} and intrinsically characterized in \cite[Lemma 1.1]{Pper} (see \S\ref{sec:sums} below). In this case, there is no distinction between regular epimorphisms as mentioned above and ordinary epimorphisms in the abelian category $\boT\modl$.

\bigskip
\section{Limit terms and exact inverse limits} \label{sec:limits}
\medskip

Let $(\boT,\mu,\varepsilon)$ be a monad on $\Sets$. We will characterize the situation where inverse limits (i.e.\ limits indexed by a downwards directed set $I$, otherwise known as the directed or cofiltered limits) preserve regular epimorphisms of $\boT$-modules. (Cf.\ the discussion of the preservation of regular epimorphisms by products in~\cite{ARV0}.) The characterization is given in terms of the existence of certain terms in~$\boT$.

To this end, let $\alpha$ be an ordinal number, which we view as usual as the set of strictly smaller ordinals. Note that for each $\beta<\alpha$, the inclusion $\alpha\setminus\beta\subset\alpha$ induces an inclusion $\boT(\alpha\setminus\beta)\hookrightarrow\boT(\alpha)$. We will, thus, identify $\boT(\alpha\setminus\beta)$ with a subset of $\boT(\alpha)$, which consists precisely of all terms $t\in\boT(\alpha)$, where the corresponding operation $t_M\colon M^\alpha\rarrow M$ does not depend, for any $M\in\boT\modl$, on the arguments with indices smaller than $\beta$ (see Lemma~\ref{lem:tails}).

We will call a term $t\in\boT(\alpha)$ is a \emph{limit term} if
\begin{enumerate}
\item[(L1)] $t\in\bigcap_{\beta<\alpha}\boT(\alpha\setminus\beta)$;
\item[(L2)] whenever $M\in\boT\modl$ and $(m_\gamma)_{\gamma<\alpha}$ with $m_\gamma=m$ for each $\gamma$ is a constant sequence in $M^\alpha$, then $t_M(m_\gamma)_{\gamma<\alpha} = m$.
\end{enumerate}

We will often be interested in mere existence of a limit term for given $\alpha$ and we will denote any such term by $\lim_\alpha\in\boT(\alpha)$. If $M$ is a $\boT$-module, we will write $\lim_{\beta<\alpha} m_\beta$ for the value of $(\lim_\alpha)_M\: M^\alpha\rarrow M$ on $(m_\beta)_{\beta<\alpha} \in M^\alpha$.

The second condition in the definition has an easy interpretation. If $t\in\boT(\alpha)$ is a term satisfying (L2), the map $\boT(1)\rarrow\boT(\alpha)$ which sends the free generator of $\boT(1)$ to $t\in\boT(\alpha)$ and which we will, by abuse of notation, denote $t$ as well, is a section to the map $\boT(c)\:\boT(\alpha)\rarrow\boT(1)$, where $c\:\alpha\rarrow1$ is the unique map of sets.

The interpretation of condition (L1) is clarified by the following lemma which implies that for any $M\in\boT\modl$, the value of $t_M\colon M^\alpha\rarrow M$ at $(m_\gamma)_{\gamma<\alpha}$ must not depend on any proper initial subsequence $(m_\gamma)_{\gamma<\beta}$ with $\beta<\alpha$.

\begin{lem} \label{lem:tails}
Let $\varnothing\ne Y\subset X$ be sets. Then a term $t\in\boT(X)$ is in the image of $\boT(Y)$ if an only if, for any $M\in\boT\modl$, the map $t_M\:M^X\rarrow M$ depends only on the components indexed by $Y$.
\end{lem}

\begin{proof}
The direct implication is clear from the discussion in Section~\ref{sec:prelim}.

For the converse, suppose that $t_M\:M^X\rarrow M$ depends only on the components from $Y$ and put $M=\boT(X)$ first. There is a canonical element $\Delta_X\in\boT(X)^X$ given by the monadic unit $\varepsilon_X\:X\rarrow\boT(X)$, and~\eqref{eq:term} together with the unitality axiom for $\boT$ implies that $t_{\boT(X)}\:\boT(X)^X\rarrow\boT(X)$ sends $\Delta_X$ to $t$.

Now choose $\Delta'\in\boT(Y)^X\subset \boT(X)^X$, which agrees with $\Delta$ on the components indexed by elements of $Y$ and such that $\Delta'_x\in\boT(Y)$ is arbitrary for $x\in X\setminus Y$. By our assumption, $t_{\boT(X)}(\Delta') = t_{\boT(X)}(\Delta) = t$. Since the inclusion $\boT(Y)\subset\boT(X)$ is a homomorphism of $\boT$-modules, we have the commutative diagram
\[
\xymatrix{
  \boT(Y)^X \ar[r]^-{t_{\boT(Y)}} \ar[d]_-{\mathsf{inc}} & \boT(Y) \ar[d]^-{\mathsf{inc}} \\
  \boT(X)^X \ar[r]^-{t_{\boT(X)}} & \boT(X)
}
\]
and $t = t_{\boT(Y)}(\Delta') \in \boT(Y)$.
\end{proof}

Combining the above considerations about the meaning of the defining conditions of (L1) and~(L2), we immediately obtain the following lemma.

\begin{lem} \label{lem:splitting}
The following are equivalent for a monad $\boT$ on $\Sets$ and an ordinal number~$\alpha$:
\begin{enumerate}
\item The map $\bigcap_{\beta<\alpha}\boT(\alpha\setminus\beta)\rarrow\boT(1)$, obtained by restricting $\boT(c)\:\boT(\alpha)\rarrow\boT(1)$ (where $c\:\alpha\rarrow1$ is the unique map in $\Sets$), is a regular epimorphism.
\item There exists a limit term $\lim_\alpha\in\boT(\alpha)$.
\end{enumerate}
\end{lem}

\begin{proof}
Since $\boT(1)$ is free on one generator, regular epimorphisms ending at $\boT(1)$ are all retractions (we can send the standard free generator of $\boT(1)$ to any preimage of it to obtain a section). 
We already know that the existence of a term $t\in\boT(\alpha)$ satisfying (L2) is equivalent to the fact that $t\:\boT(1)\rarrow\boT(\alpha)$ is a section of $\boT(c)$. Now (L1) precisely says that the image of the section is included in $\bigcap_{\beta<\alpha}\boT(\alpha\setminus\beta)$.
\end{proof}

Now we can give the main result of the section.

\begin{thm} \label{thm:ab5*-via-limits}
Let $(\boT,\mu,\varepsilon)$ be a monad on $\Sets$. Then the following conditions are equivalent:

\begin{enumerate}
\item[(i)] Inverse limits of regular epimorphisms in $\boT\modl$ are regular epimorphisms.
\item[(ii)] For each ordinal number $\alpha$, there exists a limit term $\lim_\alpha\in\boT(\alpha)$.
\end{enumerate}
\end{thm}

\begin{proof}
First of all, it follows by transfinite induction from~\cite[Lemma 1.6]{AR} that we can equivalently replace (i) by the condition:
\begin{enumerate}
\item[(i')] Well-ordered inverse limits (i.e.\ those indexed by $\alpha^\sop$, where $\alpha$ is an ordinal number) of regular epimorphisms in $\boT\modl$ are regular epimorphisms.
\end{enumerate}

First suppose that (i') holds and consider, for an ordinal number $\alpha$, the inverse system of regular epimorphisms
\begin{equation} \label{eq:key-diag}
\vcenter{
\xymatrix@C=1.6em{
  \boT(\alpha) \ar@{->>}[d]_-{\boT(c)} & \boT(\alpha\setminus 1) \ar[l]_-{\supset} \ar@{->>}[d]_-{\boT(c)} & \boT(\alpha\setminus 2) \ar[l]_-{\supset} \ar@{->>}[d]_-{\boT(c)} & \boT(\alpha\setminus 3) \ar[l]_-{\supset} \ar@{->>}[d]_-{\boT(c)} & \cdots \ar[l]_-{\supset} & \boT(\alpha\setminus\beta) \ar[l]_-{\supset} \ar@{->>}[d]_-{\boT(c)} & \cdots \ar[l]_-{\supset} \\
  \boT(1) & \boT(1) \ar[l]_-{=} & \boT(1) \ar[l]_-{=} & \boT(1) \ar[l]_-{=} & \cdots \ar[l]_-{=} & \boT(1) \ar[l]_-{=} & \cdots \ar[l]_-{=} \\
}
}
\end{equation}
By assumption, the limit map
\begin{equation} \label{eq:critical-map}
\boT(c)\: \bigcap_{\beta<\alpha} \boT(\alpha\setminus\beta) \lrarrow \boT(1)
\end{equation}
is a regular epimorphism, hence a limit term $\lim_\alpha\in\boT(\alpha)$ exists by Lemma~\ref{lem:splitting}.

Suppose conversely that (ii) holds, i.e.\ limit terms $\lim_\alpha$ exist for all ordinal numbers~$\alpha$. By Lemma~\ref{lem:splitting}, this is equivalent to the fact that \eqref{eq:critical-map} is a regular epimorphism for each $\alpha$. Yet equivalently, the following map in the diagram category $(\boT\modl)^{\alpha^\sop}$ 
\begin{equation} \label{eq:diag-split}
\vcenter{
\xymatrix@C=1.6em{
	\boT(\alpha) & \boT(\alpha\setminus 1) \ar[l]_-{\supset} & \boT(\alpha\setminus 2) \ar[l]_-{\supset} & \boT(\alpha\setminus 3) \ar[l]_-{\supset}  & \cdots \ar[l]_-{\supset} & \boT(\alpha\setminus\beta) \ar[l]_-{\supset} & \cdots \ar[l]_-{\supset} \\
	\boT(1) \ar[u]^-{\lim_\alpha} & \boT(1) \ar[l]_-{=} \ar[u]^-{\lim_\alpha} & \boT(1) \ar[l]_-{=} \ar[u]^-{\lim_\alpha} & \boT(1) \ar[l]_-{=} \ar[u]^-{\lim_\alpha} & \cdots \ar[l]_-{=} & \boT(1) \ar[l]_-{=} \ar[u]^-{\lim_\alpha} & \cdots \ar[l]_-{=} \\
}
}
\end{equation}
is a section to~\eqref{eq:key-diag}. In particular, \eqref{eq:key-diag} is a retraction in $(\boT\modl)^{\alpha^\sop}$.

Given $\beta<\alpha$ and $M\in\boT\modl$, let us denote by $\beta_!M\in(\boT\modl)^{\alpha^\sop}$ the diagram where $(\beta_!M)_\gamma=M$ for all $\gamma\le\beta$, $(\beta_!M)=\boT(0)$ is the initial object in $\boT\modl$ for $\gamma>\beta$, and the maps $(\beta_!M)_\gamma\larrow(\beta_!M)_\delta$ for $\gamma<\delta<\alpha$ are $\id_M$ or the unique maps from the initial $\boT$-module. It follows that the upper row is none other than the coproduct $\coprod_{\beta<\alpha}\beta_!\boT(1)$. On the other hand, there are obvious morphisms $\beta_!M \rarrow \beta'_!M$ for $\beta<\beta'<\alpha$ and the lower row is none other than $\colim_{\beta<\alpha}\beta_!\boT(1)$. Thus, \eqref{eq:key-diag} is simply the canonical morphism
\[ \coprod_{\beta<\alpha}\beta_!\boT(1) \lrarrow \colim_{\beta<\alpha}\beta_!\boT(1), \]
and, assuming (ii), it is a retraction.

If we now apply $\Hom_{(\boT\modl)^{\alpha^\sop}}(-,(M_\beta)_{\beta<\alpha})$ to an arbitrary diagram $(M_\beta)_{\beta<\alpha}\in(\boT\modl)^{\alpha^\sop}$, one easily checks that we obtain the underlying map in $\Sets$ of the canonical map $\lim_{\beta<\alpha} M_\beta \rarrow \prod_{\beta<\alpha} M_\beta$ between the $\boT$-modules. Assuming (ii), this map is a section of sets and $\Hom_{(\boT\modl)^{\alpha^\sop}}(-,(M_\beta)_{\beta<\alpha})$ applied to \eqref{eq:diag-split} induces a retraction (equivalently, a surjective map) of sets
\[
\prod_{\alpha<\beta} M_\beta \lrarrow \lim_{\alpha<\beta} M_\beta,
\]
which is functorial in the diagram $(M_\beta)_{\beta<\alpha}$. 

Since a product of surjective morphisms in $\Sets$ (and hence also in $\boT\modl$) is surjective, and $\lim_{\beta<\alpha}$ is a retract of $\prod_{\beta<\alpha}$ at the level of underlying sets, we get~(i').
\end{proof}

\begin{rem}
 Notice that a $\lambda$\+accessible monad (that is, a monad $\boT$ that, as
a functor $\Sets\rarrow\Sets$, preserves $\lambda$\+filtered colimits for some
regular cardinal~$\lambda$) cannot have limit terms for ordinals~$\alpha$ of
the cofinality greater or equal to~$\lambda$.
 So the monad in Theorem~\ref{thm:ab5*-via-limits} is never accessible.
 Our results thus demonstrate the usefulness of nonaccessible monads on
the category of sets (or, in other words, algebraic theories of unbounded arity
in the language of~\cite{Wr}, \cite{PR}).
\end{rem}

\begin{exs} \label{exs:compact-hausdorff}
 (1)~In any locally finitely presentable category direct limits preserve
regular monomorphisms~\cite[Corollary~1.60]{AR}.
 Hence in any locally finitely copresentable category inverse limits of
regular epimorphisms are regular epimorphisms.

 (2)~It is clear from the discussion in Section~\ref{sec:prelim} that
any category opposite to a Grothendieck abelian category $\A$ is
equivalent to the category of modules over a monad $\boT$ on~$\Sets$
(because any Grothendieck abelian category has products and
an injective cogenerator).
 Any choice of an injective cogenerator $J\in\A$ provides such a monad
$\boT$ and an equivalence $\A^\sop\simeq\boT\modl$, and vice versa.
 Since inverse limits of (regular) epimorphisms are (regular)
epimorphisms in $\A^\sop$, the monad $\boT$ satisfies the equivalent
conditions of Theorem~\ref{thm:ab5*-via-limits}.
 The discussion of additive monads on $\Sets$ satisfying these
conditions will be continued in the next Section~\ref{sec:sums}.

 (3)~In particular, the category $\B$ of compact Hausdorff abelian groups
is equivalent to the opposite category to the category $\A$ of abelian
groups (the equivalence being provided by the Pontryagin duality
functor $\Hom_\A({-},\boR/\boZ)\:\A^\sop\rarrow\B$).
 Hence the category of compact Hausdorff abelian groups is monadic
over $\Sets$; a compact Hausdorff abelian group is the same thing as
a module over the monad $\boT$ on $\Sets$ assigning to a set $X$
the set of all (discontinuous) abelian group homomorphisms
$(\boR/\boZ)^X\rarrow\boR/\boZ$. Indeed, $J=\boR/\boZ$ is an injective
cogenerator of $\A$, so that $P=\Hom_\A(J,\boR/\boZ)$ is
a projective generator of $\B$ and we have
\[ \boT(X) = \Hom_\B(P,P^{(X)}) = \Hom_\A(J^X,J). \]
 According to~(2), this monad satisfies the equivalent conditions of
Theorem~\ref{thm:ab5*-via-limits}.

 (4) Similar considerations apply to the correspondence between locally
finite Grothendieck categories and categories of pseudocompact modules
\cite[\S IV.3]{Gab}. Let $R$ be a complete separated topological ring
with a basis of neighborhoods of $0$ formed by left ideals $I$ such that
$R/I$ is a module of finite length and consider the category $\B$ of
complete separated topological left $R$-modules $M$ with a basis
of submodules $N$ such that $M/N$ has finite length. Then $\A=\B^\sop$ is
a locally finite Grothendieck category with injective cogenerator $J=R^\sop$
and every locally finite Grothendieck category arises in this way
(see also \cite[Corollaire 7.2]{Jen} for additional details). In particular,
$\B$ is monadic and satisfies the equivalent conditions of
Theorem~\ref{thm:ab5*-via-limits}.

As a particular example, let $R=k$ be a field with the discrete topology. Then
the category $\B$ is the category of complete separated topological vector
spaces with a basis of neighborhoods of 0 formed by subspaces of finite codimension.
The corresponding monad is given by $\boT(X) = \Hom_k(k^X,k)$.

 (5)~The category of compact Hausdorff topological spaces is
locally copresentable, being equivalent to the opposite category to
the category of commutative unital $C^*$\+algebras (but it is not
locally finitely copresentable).
 The category of compact Hausdorff topological spaces is also monadic
over $\Sets$; a compact Hausdorff topological space is the same thing
as a module over the monad $\boT$ on $\Sets$ assigning to a set $X$
the underlying set of the Stone--\v Cech compactification of $X$,
that is, the set of all ultrafilters on~$X$
\cite[Example~20.36(1)]{AHS}.
 Using Tychonoff's product compactness theorem, one easily shows that
inverse limits of regular epimorphisms (\,$=$~surjective continuous
maps) of compact Hausdorff topological spaces are regular epimorphisms.
 Thus the monad $\boT$ satisfies the equivalent conditions of
Theorem~\ref{thm:ab5*-via-limits}.
\end{exs}

\bigskip
\section{Infinite summation versus limit terms} \label{sec:sums}
\medskip

In the last section, we focus on abelian categories $\B$ with coproducts and a projective generator $P$, or equivalently on the module categories of additive monads $\boT$ on sets (in particular, we can assume $\B=\boT\modl$ and $P=\boT(1)$). In that case, we will prove that  $\B$ satisfies Grothendieck's condition Ab5* if and only if for any set $X$, the canonical homomorphism $P^{(X)} \rarrow P^X$ is surjective. In the opposite category $\A=\B^\sop$, it will follow that the Ab5 property is equivalent to a condition on the object $J=P^\sop\in\A$ which in ordinary module categories characterizes the pure-injectivity of $J$ (see \cite[Theorem 7.1(vi)]{JL}).

\smallskip

An additive monad $\boT$ is characterized by the fact that it contains terms $+\in\boT(2)$, $-\in\boT(1)$ and $0\in\boT(0)$ which satisfy the usual axioms for abelian groups and commute with all other terms (see \cite[\S1]{Pper} for details). By the monadic composition, we obtain for every natural number $n$ a term $\sum_n\in\boT(n)$ which is, for any $\boT$-module $M$, given by $\left(\sum_n\right)_M(m_0,m_1,\dots,m_{n-1}) = m_0+m_1+\cdots+m_{n-1}$.

As we defined limit terms in the last section, we can, for additive monads, also formally define terms of infinite summation. Given a set $X$, a term $t\in\boT(X)$ is a \emph{summation term} if for each $\boT$-module $M$ and a family of elements $(m_x)_{x\in X}\in M^X$ with only finitely many (pairwise distinct) indices $x_0$, $x_1$,~\dots, $x_{n-1}\in X$ for which $m_{x_i}\ne 0_M$, we have $t_M(m_x)_{x\in X} = m_{x_0}+m_{x_1}+\cdots+m_{x_{n-1}}$.

Summation terms can be also characterized syntactically.

\begin{lem}\label{lem:sum-syntax}
The following are equivalent for $t\in\boT(X)$, where $\boT$ is an additive monad on $\Sets$:
\begin{enumerate}
\item The image of $t$ under the canonical $\boT$-module homomorphism $\boT(X)=\boT(1)^{(X)}\rarrow\boT(1)^X$ is a constant family $(m_x)_{x\in X}$, where $m_x$ is the standard free generator of $\boT(1)$ for each $x\in X$.
\item $t$ is a summation term.
\end{enumerate}
\end{lem}

\begin{proof}
For each $x\in X$, denote by $p_x\:\boT(X)\rarrow\boT(1)$ the composition of the canonical map from (1) with the product projection to the component indexed by~$x$. This map sends the basis element $e_x$ of $\boT(X)$ corresponding to~$x$ to the basis element of $\boT(1)$ and all other basis elements $e_y$, $y\in X$, $y\ne x$ to $0$, and assertion (1) is clearly equivalent to

\begin{enumerate}
\item[(1')] The map $p_x$ is a retraction of $t\:\boT(1)\rarrow\boT(X)$ for every $x\in X$.
\end{enumerate}

If (1') holds and $M\in\boT\modl$, then $\Hom_{\boT\modl}(p_x,M)\:M\rarrow M^X$ is a section of $t_M\:M^X\rarrow M$ and sends $m\in M$ to $(m_y)_{y\in X}$ such that $m_y=0$ for $y\ne x$ and $m_x=m$. In particular $t_M(m_y)_{y\in X}=m$ and, since $t_M$ commutes with + by the characterization of additive monads, $t$ is a summation term.
If, conversely, $t$ is a summation term, then the inclusion $\iota_x\:\boT(1)\rarrow\boT(1)^X$ to the component~$x$ composes with $t_{\boT(1)}\:\boT(1)^X\rarrow\boT(1)$ to $\id_{\boT(1)}$. However, this composition is obtained from the composition $p_x\circ t$ by the application of $\Hom_{\boT\modl}(-,\boT(1))$, and (1') follows.
\end{proof}

In particular, let $X=\alpha$ be an ordinal.
We will denote any fixed infinite summation term by $\sum_\alpha\in\boT(\alpha)$ (this is consistent with the above notation for natural numbers) and we will write $\sum_{\beta<\alpha} m_\beta$ for the value of $\left(\sum_\alpha\right)_M\:M^\alpha\rarrow M$ on a sequence $(m_\beta)_{\beta<\alpha}\in M^\alpha$. More generally, we will write $\sum_{\beta\le\gamma<\alpha} m_\gamma$ for the value of $\sum_\alpha$ on a sequence $(m_\gamma)_{\gamma<\alpha}\in M^\alpha$ with $m_\gamma=0$ for all $\gamma<\beta$.

Note that if $\beta<\alpha$ are ordinal numbers and $\sum_\alpha\in\boT(\alpha)$, we can define a term $\sum_\alpha\!|_\beta\in\boT(\beta)$ such that on each $M\in\boT\modl$, $\left(\sum_\alpha\!|_\beta\right)_M\: M^\beta\rarrow M$ acts via the composition
\[ M^\beta \overset{\iota}\lrarrow M^\alpha \overset{\sum_\alpha}\lrarrow M, \]
where $\iota\:M^\beta\rarrow M^\alpha$ is the inclusion to the components below $\beta$. Clearly, $\sum_\alpha\!|_\beta$ is a summation term for $\beta$. Hence, if we fix $\alpha$ and $\sum_\alpha$ exists, we can always choose summation terms for all $\beta<\alpha$ so that we have $\sum_\beta = \sum_\alpha\!|_\beta$ for each $\beta<\alpha$. In this case, we say that the infinite summations are \emph{compatible below~$\alpha$}. 

A key observation for our application is that for additive monads, the existence of summation terms is equivalent to the existence of limit terms.

\begin{prop} \label{prop:sums-vs-lim}
Let $\boT$ be an additive monad and $\alpha$ an ordinal number. Then the following assertions are equivalent:

\begin{enumerate}
\item[(i)] a summation term $\sum_\alpha\in\boT(\alpha)$ exists,
\item[(ii)] limit terms $\lim_\beta\in\boT(\beta)$ exists for all $\beta\le\alpha$.
\end{enumerate}
\end{prop}

\begin{proof}
Starting from (i), we assume by the preceding discussion that summation terms $\sum_\beta$ exist and are compatible below $\alpha$.
We will construct the limit terms $\lim_\beta$, $\beta\le\alpha$ inductively. We formally put $\lim_0 = 0 \in \boT(0)$ and if the limit terms have been defined for all $\gamma<\beta$ for some $\beta>0$, we define $\lim_\beta\in\boT(\beta)$ so that for each $M\in\boT\modl$ and $(m_\gamma)_{\gamma<\beta}\in M^\beta$ we have
\[ \lim_{\gamma<\beta} m_\gamma = \sum_{\gamma<\beta} \bigl(m_\gamma - \lim_{\delta<\gamma} m_\delta\bigr). \]
%

Along with the construction, we need to prove that $\lim_\beta$ satisfies (L1) and (L2) from the previous section. To this end, observe that if $\gamma<\beta\le\alpha$, then
\begin{align*}
\lim_{\delta<\beta} m_\delta &= \sum_{\delta<\beta}\bigl(m_\delta - \lim_{\zeta<\delta} m_\zeta\bigr) \\
&= \left(\sum_{\delta<\gamma}\bigl(m_\delta - \lim_{\zeta<\delta} m_\zeta\bigr)\right) + \left(\sum_{\gamma\le\delta<\beta}\bigl(m_\delta - \lim_{\zeta<\delta} m_\zeta\bigr)\right) \\
&= \lim_{\delta<\gamma} m_\delta + \left(\sum_{\gamma\le\delta<\beta}\bigl(m_\delta - \lim_{\zeta<\delta} m_\zeta\bigr)\right),
\end{align*}
where the second equality uses the compatibility of the summations below $\alpha$.

If $\beta=\gamma+1$ is an ordinal successor, then the latter computation specializes to $\lim_{\delta<\beta} m_\delta = \lim_{\delta<\gamma} m_\delta + (m_\gamma - \lim_{\zeta<\gamma} m_\zeta) = m_\gamma$. Using Lemma~\ref{lem:tails}, this clearly implies (L1) and (L2) for $\lim_\beta$.

Suppose on the other hand that $\beta$ is a limit ordinal. If $\gamma<\beta$, then also $\gamma+1<\beta$ and we have
\begin{align*}
\lim_{\delta<\beta} m_\delta
&= \lim_{\delta<\gamma+1} m_\delta + \left(\sum_{\gamma+1\le\delta<\beta}\bigl(m_\delta - \lim_{\zeta<\delta} m_\zeta\bigr)\right) \\
&= m_\gamma + \left(\sum_{\gamma+1\le\delta<\beta}\bigl(m_\delta - \lim_{\zeta<\delta} m_\zeta\bigr)\right).
\end{align*}
Then Lemma~\ref{lem:tails} and the inductive hypothesis on the $\lim_\delta$, $\gamma+1\le\delta<\beta$, implies that $\lim_\beta\in\boT(\beta\setminus\gamma)$. This implies (L1) for $\lim_\beta$. To show (L2), suppose that $(m_\gamma)_{\gamma<\beta}$ is a constant sequence in $M^\beta$. Since $m_\gamma-\lim_{\delta<\gamma}m_\delta = 0$ for all $0<\gamma<\beta$ by the inductive hypothesis, we obtain $\lim_{\gamma<\beta} m_\gamma = \sum_{\gamma<\beta} (m_\gamma - \lim_{\delta<\gamma} m_\delta) = m_0$, as required.

Let us conversely start from (ii). We will construct $\sum_\beta$ by induction on $\beta\le\alpha$. We again start with $\sum_0=0\in\boT(0)$ and if $\beta=\gamma+1$ is an ordinal successor, we put $\sum_{\delta<\beta} m_\delta = \bigl(\sum_{\delta<\gamma} m_\delta\bigr) + m_\gamma$. The latter is clearly a summation term provided that $\sum_\gamma$ is. For limit ordinals, we use the formula
\[
\sum_{\gamma<\beta} m_\gamma = \lim_{\gamma<\beta} \left(\sum_{\delta<\gamma} m_\delta\right).
\]
Given any $(m_\gamma)_{\gamma<\beta}\in M^\beta$ with only finitely many indices $\gamma_0<\gamma_1<\cdots<\gamma_{n-1}<\beta$ for which $m_{\gamma_i}\ne 0$, the sequence $\bigl(s_\gamma=\sum_{\delta<\gamma} m_\delta\bigr)_{\gamma<\beta}$ satisfies by the inductive hypothesis $s_\gamma = m_{\gamma_0} + m_{\gamma_1} + \cdots + m_{\gamma_{n-1}}$ whenever $\gamma_{n-1}<\gamma<\beta$. Hence also $\sum_{\gamma<\beta} m_\gamma = m_{\gamma_0} + m_{\gamma_1} + \cdots + m_{\gamma_{n-1}}$ by (L1) and (L2).
\end{proof}

Now we can prove the main result which was announced in the introduction.

\begin{thm}\label{thm:ab5*-via-sums}
Let $\B$ be an abelian category with coproducts and a projective generator $P$. Then the following conditions are equivalent:
\begin{enumerate}
\item[(i)] Inverse limits are exact in $\B$.
\item[(ii)] The canonical map $P^{(X)} \rarrow P^X$ is surjective for every set $X$.
\item[(iii)] The diagonal map $\Delta_X\:P\rarrow P^X$ factors through $P^{(X)} \rarrow P^X$ for every $X$.
\end{enumerate}
\end{thm}

\begin{proof}
We can without loss of generality assume that $\B=\boT\modl$ and $P=\boT(1)$, where $\boT$ is an additive monad on $\Sets$. Condition (i) is equivalent to the existence of a limit term $\lim_\alpha$ for each ordinal number $\alpha$ by Theorem~\ref{thm:ab5*-via-limits}, while condition (iii) is equivalent to the existence of a summation term $\sum_X$ for every set $X$ by Lemma~\ref{lem:sum-syntax}. The equivalence between (i) and (iii) then immediately follows from Proposition~\ref{prop:sums-vs-lim}.

Since $P$ is projective, (ii) implies (iii). The converse was discussed already in the introduction. If $f=(f_x)\:P\rarrow P^X$ is any morphism in $\B$, where $f_x\:P\rarrow P$ denote its components, then $f$ factorizes to the composition in the lower row of the following diagram,
\[
\xymatrix{
& P^{(X)} \ar[d] \ar[r]^-{\coprod f_x} & P^{(X)} \ar[d] \\
P \ar[r]_-{\Delta_X} \ar@{.>}[ur] & P^X \ar[r]_-{\prod f_x} & P^X.
}
\]
If $\Delta_X$ factors through the map $P^{(X)} \rarrow P^X$, so does $f$.
\end{proof}

\end{document}